\documentclass[runningheads]{llncs}

\usepackage{amsthm}
\usepackage{multirow}

\usepackage[T1]{fontenc}
\usepackage[english]{babel}       
\usepackage[utf8]{inputenc}       
\usepackage{hyperref}

\usepackage{amsmath,amsthm,amssymb}

\usepackage{algorithm}
\usepackage{algpseudocode}

\usepackage{subcaption}						
\usepackage{tabularx}
\usepackage{standalone}		
\usepackage{pdfpages}
\usepackage{tikz}
\usepackage{tcolorbox}
\newtcolorbox{mybox}{colback=blue!5!white,
colframe=blue!75!black}

\usepackage{bbm}

\theoremstyle{plain} 

\theoremstyle{definition}

\theoremstyle{remark}

  \usetikzlibrary{external}       
  \tikzsetexternalprefix{TikZ/}
\usepackage{pgfplots}             
  \pgfplotsset{                   
        table/search path={pics},
    }
    
\usepackage{etoolbox}

\usepackage[numbers,sort&compress]{natbib}
\bibpunct[, ]{[}{]}{,}{n}{,}{,}
\begin{document}

\title{Presolve techniques for quasi-convex \\chance constraints with finite-support low-dimensional uncertainty}
%
%
\author{Guillaume Van Dessel\inst{1}\orcidID{0000-0002-7302-8579} \and \\
Fran\c{c}ois Glineur\inst{1,2}\orcidID{0000-0002-5890-1093}}
\authorrunning{G. Van Dessel \& F. Glineur}
%
\institute{UCLouvain, ICTEAM (INMA), 4 Avenue Georges Lema\^{\i}tre, Louvain-la-Neuve, BE \and
UCLouvain, CORE, 34
Voie du Roman Pays, Louvain-la-Neuve, BE}
\maketitle  

\begin{abstract}
Chance-constrained programs (CCP) represent a trade-off between conservatism and robustness in optimization. In many CCPs, one optimizes an objective under a probabilistic constraint continuously parameterized by a random vector $\xi$. In this work, we study the specific case where the constraint is quasi-convex with $\xi$. Moreover, the support of vector $\xi$ is a collection of $N$ scenarios in dimension $p=2$ or $p=3$.\\ In general, even when both the constraint and the objective are convex in the decision variable, the feasible region of a CCP is nonconvex, turning it into a difficult problem. However, under mild assumptions, many CCPs can be recast as big-$M$ mixed-integer convex programs (MICP). \\Unfortunately, the difficulty of these MICPs explodes with the number of scenarios, restricting the instances practically solvable in decent time. \\
To cut down the effective number of scenarios considered in MICP reformulations and accelerate their solving, we propose and test presolve techniques based on computational geometry. Our techniques produce certificates to discard or select \emph{a priori} some scenarios before solving a regular MICP.  Moreover, the information aggregated during presolve leverages the possibility to strengthen big-$M$ constants. Our numerical experiments suggest that spending some time in presolve is more efficient than a direct solve for a class of \emph{probabilistic projection} problems, including an interesting type of \emph{facility location} problem.
\end{abstract}

\section{Introduction}
Let $\xi$ be a $p$-dimensional random vector with finite support of cardinality $N$, i.e. $\textbf{supp}(\xi) = \{\xi^{(1)},\dots,\xi^{(N)}\}$.
We are concerned with optimization problems of the form
\begin{equation}
  F^*_\tau := \min_{x\,\in\,\mathcal{X}}\,F(x) \quad \text{s.t.}\quad \mathbb{P}_\xi[c(x,\xi)\leq 0]\geq 1-\tau \label{eq:min_ccp_gen} \tag{CCP}
\end{equation}
where $F$ represents the objective function to be minimized, $\xi$ parametrizes the probabilistic constraint, $\mathcal{X}\subseteq \mathbb{R}^d$ encodes deterministic constraints that must be satisfied (i.e. independently from $\xi$) and $\tau \in (0,1)$ is a \emph{relaxation tolerance} with respect to the (fully deterministic) robust counterpart of \eqref{eq:min_ccp_gen}, i.e.
\begin{equation}
F^*_0 := \min_{x\,\in\,\mathcal{X}}\,F(x) \quad \text{s.t.}\quad c(x,\xi^{(s)})\leq 0 \quad \forall s \in\{1,\dots,N\}.\label{eq:min_ro_gen} \tag{RO}
\end{equation}
Given the finite support of $\xi$, one usually rewrites \eqref{eq:min_ccp_gen} as follows \vspace*{-.4cm}
\begin{equation}
  F^*_\tau = \min_{x\,\in\,\mathcal{X}}\,F(x) \quad \text{s.t.}\quad \sum_{s=1}^N\,\pi^{(s)} \cdot \mathbbm{1}\big(c(x,\xi^{(s)})\leq 0\big)\geq 1-\tau. \label{eq:min_ccp_gen_2} 
\end{equation}\vspace*{-.2cm}
with $\pi^{(s)}:=\mathbb{P}[\xi = \xi^{(s)}]$ for every $s \in [N]:=\{1,\dots,N\}$.\\

\noindent Problems like \eqref{eq:min_ccp_gen} arise in various fields, e.g. in finance \cite{Swain23}, supply-chain management \cite{Chen01}, optimal vaccination strategies \cite{Tanner08} or even microgrid scheduling \cite{Liu17}. While robustness is arguably a desirable feature of solutions of optimization problems, formulations like \eqref{eq:min_ro_gen} tend to be overly conservative. Indeed, their optimal solutions are affected by every outcome of the random parameters $\xi$, even by the rarest scenarios that, by definition, do not occur often in practice. \\
Note that it can happen that \eqref{eq:min_ro_gen} is not even feasible, especially in the case of infinite supports. Setting the \emph{relaxation tolerance} $\tau>0$ (sometimes referred to as \emph{risk tolerance} in the literature, see e.g. \cite{Ahmed17})  big enough can turn problem $F^*_\tau$ into a feasible one yet keeping a fair amount of robustness by considering most of constraints induced by the empirical scenarios. Solutions of \eqref{eq:min_ccp_gen} can dramatically improve the cost, i.e. $F^*_\tau \ll F^*_0$, but at the price of solving a nonconvex problem in the general case (see Section \ref{sec:feasible_set_ccp}). 
\subsection*{Big-M approach}
\noindent In the present context, one usually invokes the possibility to compute, for any realization of $\xi$, a so-called big-$M$ bound. For every $s \in [N]$, one has access to
\begin{equation}
M^{(s)} \geq \max_{x\,\in\,\mathcal{X}}\,c(x,\xi^{(s)}). \label{eq:raw_bigMs}
\end{equation}
Based upon \eqref{eq:min_ccp_gen_2} and the above big-$M$ constants \eqref{eq:raw_bigMs}, one can introduce $N$ \emph{switching} binary variables $z=(z^{(1)},\dots,z^{(N)})$ to reformulate \eqref{eq:min_ccp_gen} as a mixed-integer program \eqref{eq:new_min_problem} (see e.g. \cite{Ahmed17,Kucukyavuz22, Roland23} and references therein) \vspace*{-.2cm}
\begin{align}
F^*_\tau = \min_{x\, \in\, \mathcal{X},\,z\,\in\,\{0,1\}^N} \quad & F(x) \label{eq:new_min_problem} \tag{MI-CCP}
 \\
\textrm{s.t.} \quad & \sum_{s=1}^N\,\pi^{(s)}\,z_s \geq 1-\tau \nonumber \\
 & c(x,\xi^{(s)})\leq M^{(s)} \cdot (1-z_s) \quad \forall s \in [N].\nonumber
\end{align}
As such, \eqref{eq:new_min_problem} may be computationally expensive to solve.  Notably, it is well-known that the large number of possibly quite conservative big-$M$ constants leads to poor continuous relaxations \cite{Camm90}.
Yet, for many applications of interest, additional structure (e.g. \cite{deOliveira25} wherein $c(x,\xi)$ is linear in $\xi$) is available. This leverages heuristics that allow to efficiently compute lower-bounds $\check{F}^*_\tau\leq F^*_\tau$ and upper-bounds $\hat{F}^*_\tau\geq F^*_\tau$ as well as valid inequalities, substantially cutting down the effective search-space of \eqref{eq:new_min_problem} during a presolve stage. We investigate this approach under the hypothesis that $c(x,\xi)$ is quasi-convex in $\xi$. 
\subsection*{Goals} As already claimed, \eqref{eq:new_min_problem} tractability crucially depends on how large are the big-$M$ values and how many scenarios (hence binary variables) are considered in the model. Unfortunately, even when the inequalities \eqref{eq:raw_bigMs} are tight, the structure of $c$ as well as the diameter of $\mathcal{X}$ can both lead to arbitrarily large bounds. Thereby, in order to speed-up the solving of \eqref{eq:min_ccp_gen} (which is our main goal), one can proceed as follows. \\ \vspace{-5pt}

\smallskip \noindent -- \textbf{Logic encoding}. First, based on topological arguments in the $\xi$-space, we devise logical constraints involving variables $z$ that induce a partition of the scenarios into a \emph{safe set} $\oplus \subseteq [N]$, a \emph{pruned set} $\ominus \subseteq [N]$ and a \emph{selectable set} $[N]\backslash (\oplus \cup \ominus)$. Let $F^*_\tau$ be finite and let $(x^*,z^*)$ be an optimal solution of \eqref{eq:new_min_problem}. If $s \in \oplus$ then either $z^*_s=1$ or $z^*_s$ can be set to $1$ without changing the optimal value $F^*_\tau$. If $s \in \ominus$, $z_s^*$ must be equal to $0$, i.e. the scenario $\xi^{(s)}$ can be ignored. Finally, inspired by \cite{Ruszczynski02},  we provide valid inequalities that couple the values $z^*_s$ for $s \in [N]\backslash (\oplus \cup \ominus)$ so that at least one of them must be $0$. 
\\

\smallskip \noindent --   \textbf{big-$M$ tightening}.  
    Second, akin to \cite{Roland23}, we tighten the big-$M$ values based on $\oplus$ as well as, possibly, an upper-bound $\hat{F}^*_\tau \geq F^*_\tau$ (e.g. $\hat{F}^*_\tau = F^*_0$ if \eqref{eq:min_ro_gen} is \emph{feasible}).  For every $s \in [N]\backslash (\oplus \cup \ominus)$, we compute 
    \begin{equation}
M^{(s)}(\oplus,\hat{F}^*_\tau) \geq \max_{x\,\in\,\mathcal{X},\,F(x)\leq \hat{F}^*_\tau}\,c(x,\xi^{(s)})\quad \text{s.t.}\quad c(x,\xi^{(\tilde{s})})\leq 0\quad\forall \tilde{s} \in \oplus. \label{eq:refined_bigMs_safe}
\end{equation}
Note that if the problem at the right-hand side of \eqref{eq:refined_bigMs_safe} admits an optimal value less or equal than $0$, the scenario $s$ can be added to the \emph{safe set} $\oplus$. \\Regarding \eqref{eq:raw_bigMs}, $M^{(s)}$ is a shorthand for $M^{(s)}(\emptyset, +\infty)$ as defined in \eqref{eq:refined_bigMs_safe}.
\vspace{-3pt}

\subsection*{Notations} When an optimization problem is \emph{infeasible}, i.e. its \emph{feasible set} is empty, we assign it the optimal value $+\infty$. Conversely, when it is \emph{unbounded}, e.g. there exists a \emph{recession direction} along which the objective is not lower-bounded, we set the optimal value to $-\infty$. Let $\bar{\xi} \in \mathbb{R}^p$ denote a realization of $\xi$, we introduce \begin{equation}\mathcal{R}(\bar{\xi}) := \{x \in \mathbb{R}^d\,|\,c(x,\bar{\xi})\leq 0\} \label{eq:region_scenario} \end{equation}
as the set of admissible decision vectors for the probabilistic constraint under $\bar{\xi}$. 
For every subset $S \subseteq [N]$, we further define the subproblem
\begin{equation}
\nu(S) :=  \min_{x\,\in\,\mathcal{X}}\,F(x) \quad \text{s.t.}\quad c(x,\xi^{(s)})\leq 0 \quad \forall s \in S. \label{eq:min_S_ccp} \tag{$S$-subproblem}
\vspace{-7pt}
\end{equation} 
 \subsection*{Assumptions}
We lay down some useful assumptions that will hold throughout the sequel. \vspace*{-.2cm}
\begin{itemize} 
\item[(A)] For every $s \in [N]$, $\pi^{(s)}>0$ and $\nu(\{s\})>-\infty$.
\item[(B)] For every $x \in \mathcal{X}$, the function $c(x,\cdot)$ is quasi-convex.
\end{itemize}
\begin{remark}
 As pointed out in \cite{Luedtke14}, Assumption (A) is without loss of generality. Indeed, if either $\pi^{(s)}=0$ or $\mathcal{X} \,\cap\, \mathcal{R}(\xi^{(s)})= \emptyset$, $\xi^{(s)}$ should be removed from the dataset of considered scenarios and the problem becomes $F^*_{\tilde{\tau}}$ with $\tilde{N}=N-1$ scenarios and $\tilde{\tau}=\tau-\pi^{(s)}$. Note that this assumption implies the \emph{boundedness} of $F^*_\tau$ when $\tau<1$ since there must be at least one $s\in [N]$ such that $z^*_s=1$ at an optimal solution $(x^*,z^*)$. On the other hand, \emph{feasibility} might be quite difficult to check \emph{a priori}. Assumption (B) represents the building block of our presolve heuristics, i.e. we extensively use the convexity of the sublevel-sets of $c(x,\cdot)$ to add scenarios into $\oplus$ and $\ominus$, see \textbf{logic-encoding}. 
 \end{remark} 

\noindent We end this section with the illustrative Example \ref{example:pfl} \& \ref{example:to_find} fulfilling our assumptions. \vspace*{-.5cm}

\subsubsection*{Probabilistic Ball Projection} 
We are provided $N$ different spatial locations, i.e. the rows of the data matrix $\mathbf{D} = (\xi^{(1)},\dots,\xi^{(N)})^T \in \mathbb{R}^{N \times p}$, a reference point in space $\bar{x} \in \mathbb{R}^p$ and a pair of norms $(\|\cdot \|_{o},\|\cdot\|_{\tilde{o}})$ with $o,\tilde{o} \in \{1,2,\infty\}$. The goal is to find the closest point $x\in \mathbb{R}^p$ from $\bar{x}$, i.e. minimizing the distance $\|x-\bar{x}\|_o$ while ensuring that it belongs to $\mathbb{B}_{\tilde{o}}(\xi,R)$ with probability of at least $1-\tau$. Assuming that $\mathbb{P}[\xi = \xi^{(s)}]=\pi^{(s)}>0$ for every $s\in [N]$, the problem reads
\begin{equation}
F^*_\tau = \min_{x \in \mathcal{X}}\, \underbrace{\|x-\bar{x}\|_o}_{F(x)}\quad \text{s.t.}\quad \mathbb{P}_{\xi}[\underbrace{\|x-\xi^{(s)}\|_{\tilde{o}}-R}_{c(x,\xi)} \leq 0]   \geq 1-\tau.
\label{eq:pfl_ccp} \tag{PBP-($p,o,\tilde{o}$)}
\end{equation}
Here, $p=d$ and $\mathcal{X} = \mathbb{B}_{\infty}(\mathbf{0}_d, \bar{R})$ is as uniform box in $\mathbb{R}^d$ with $\bar{R}>0$. \\
For any $s\in[N]$, $\mathcal{X}(\{s\}):=\mathcal{X}\,\cap\,\mathcal{R}(\xi^{(s)})= \mathbb{B}_{\tilde{o}}(\xi^{(s)},R)$ is convex and compact. Assumption (A) is satisfied with $\nu(\{s\}) = \|\textbf{proj}_{\mathcal{X}(\{s\})}(\bar{x})-\bar{x}\|_o\geq0>-\infty$ and $c(x,\xi)$ is convex (thus quasi-convex) in $\xi$ so that Assumption (B) is also verified.

\noindent We present hereafter an instance of PBP-($2,2,1$) (Example \ref{example:pfl}) with real application and one instance of PBP-$(3,2,\infty)$ (Example \ref{example:to_find}), less conventional we admit. 

\begin{example}[Probabilistic Facility Location]\label{example:pfl}
\\
\vspace{-7pt}

\noindent  One must find where to install a new heliport facility $x \in \mathbb{R}^2$ so to minimize its flying $L_2-$distance with respect to a reference hospital $\bar{x} \in \mathbb{R}^2$. The concerned population living in the surroundings is aggregated at $N$ different locations, i.e. $(\xi^{(1)},\dots,\xi^{(N)})$. A proportion $\pi^{(s)} >0$ of the population is located at position $\xi^{(s)}$ for every $s\in [N]$. The design constraint requires that an emergency happening at random among the population should be located within a radius of $R>0$ in Manhattan $L_1-$distance from the heliport with probability of a least $1-\tau$. 
\end{example}

\begin{example}[Optimal Watch Spot]\label{example:to_find}
\\
\vspace{-7pt}

\noindent In the same spirit, one can think about the following problem. One would like to be closest from $\bar{x} \in \mathbb{R}^3$ to watch an event happening there while being in a safe place with probability of at least $1-\tau$. For every $s\in [S]$, there is a $\pi^{(s)}\in (0,1)$ chance that the outcome of a random meteorological phenomenon implies the following: $\xi^{(s)}$ depicts the center of 3D uniform box of radius $R>0$ wherein one needs to stay to safely watch the aforementioned event.
\end{example}


\section{Enumerative complexity}\label{sec:feasible_set_ccp}
In this section, we show that \eqref{eq:new_min_problem} can be posed as the global minimization of an objective function $\mathcal{F}_\tau$ stemming as the minimum of a finite collection of functions. As we will unveil, the domain of these functions differ, yielding an objective $\mathcal{F}_\tau$ not everywhere continuous on $\mathcal{X}$. Recent works \cite{VanDessel24,VanDessel25} are devoted to tackling the global minimization of a minimum of a finite collection of functions. Unfortunately,  the methods presented therein are not applicable since they crucially depend on the continuity of the objective as a whole. To obtain a formulation wherein the size of the aforementioned collection is minimal, we underpin the following fact. Without loss of generality, there will always exist an optimal solution $(x^*,z^*)$ to \eqref{eq:new_min_problem}, called \emph{minimal solution}, so that \begin{equation}
    \sum_{s\in[N]}\,\pi^{(s)}z^*_s - \min_{\tilde{s}\in \{s \in [N]\,|\, z^*_s=1\}}\,\pi^{(\tilde{s})} < 1-\tau. \label{eq:consistency_minimal_solution}
\end{equation} 
Indeed, if \eqref{eq:consistency_minimal_solution} is not satisfied, $(x^*,z^*-\mathbf{e}_{\tilde{s}})$ becomes admissible and optimal for \eqref{eq:new_min_problem} and we can proceed recursively until \eqref{eq:consistency_minimal_solution} holds. We note that such $x^*$ is closely related to the concept of $(1-\tau)$ efficient point \cite{Dentcheva00,Lejeune12}.  Every optimal solution $(x^*,z^*)$ induces a \emph{selection} $\hat{S}(z^*) = \{s \in [N]\,|\,z_s^*=1\}\subseteq [N]$ of scenarios so that $F^*_\tau = F(x^*) = \nu(\hat{S}(z^*))$, recalling \eqref{eq:min_S_ccp}. \\ \\We introduce now Definition \ref{def:sub_int} describing the structure of the induced \emph{selections} by \emph{minimal solutions}. 
\begin{definition}[Minimal subset for \eqref{eq:new_min_problem}]
Let $S \subseteq [N]$. $S$ is called a minimal subset for \eqref{eq:new_min_problem} if and only if 
\begin{equation}
   1-\tau \leq \sum_{s\in S}\,\pi^{(s)}  < 1-\tau + \min_{\tilde{s}\in S}\,\pi^{(\tilde{s})}. \label{eq:consistency_minimal_solution_def}
\end{equation} 
\label{def:sub_int}
\end{definition}
\begin{remark}
Obviously, if $(x^*,z^*)$ is a \emph{minimal solution} then $\hat{S}(z^*)$ is a \emph{minimal subset}. However, if $S$ is a \emph{minimal subset} then the minimizer $x^*(S)$ of $\nu(S)$ in \eqref{eq:min_S_ccp} is not necessarily a global minimizer of \eqref{eq:new_min_problem}.
\end{remark}

\noindent We depict by $\mathbf{S}_\tau$ the collection containing every \emph{minimal subset} for \eqref{eq:new_min_problem}. Let us now highlight in Example \ref{example:equiprobable} a first clue of the combinatorial nature of \eqref{eq:new_min_problem}. We express $\mathbf{S}_\tau$ in explicit form when every scenario is equiprobable.
\begin{example}[Equiprobable scenarios] \label{example:equiprobable}\\
\vspace{-7pt}

\noindent Let $\pi^{(s)}=1/N$ for every $s \in [N]$. It comes that \eqref{eq:consistency_minimal_solution_def} can be rephrased as 
\begin{equation}
   N\cdot(1-\tau) \leq |S| < N\cdot(1-\tau) + 1. \label{eq:equiprobable_scenarios}
\end{equation}
From \eqref{eq:equiprobable_scenarios}, one can deduce that $\mathbf{S}_\tau = \big\{ S \subseteq [N]\,|\, |S| = \lceil N\cdot(1-\tau)\rceil\big\}$ so that 
\begin{equation}
\big| \mathbf{S}_\tau\big| = \binom{N}{\lceil N\cdot(1-\tau)\rceil}. \label{eq:n_equiprobable_scenarios}
\end{equation}
\end{example}

\noindent Since (at least) one global minimizer $(x^*,z^*)$ of \eqref{eq:new_min_problem} is \emph{minimal} and can be obtained by setting $x^*$ as the minimizer of $\nu(\hat{S}(z^*))$ with $\hat{S}(z^*) \in \mathbf{S}_\tau$ then a strategy to solve \eqref{eq:new_min_problem} can work as follows. One iterates over $\mathbf{S}_\tau$, trying one \emph{minimal subset} $S$ at the time by computing $\nu(S)$. Ultimately, it comes
\begin{equation}
F^*_\tau = \min_{S\,\in\,\mathbf{S}_\tau}\,\nu(S). \label{eq:enumeration_ccp}
\end{equation}
\begin{remark}
\vspace{-10pt}
When all the scenarios are equally as likely (Example \ref{example:equiprobable}), the exhaustive enumeration requested in \eqref{eq:enumeration_ccp} is straightforward to implement. Under more general distributions, one can solve an incremental sequence of feasibility problems. Starting with an empty collection $\hat{\mathbf{S}}_\tau = \emptyset$, we compute the next \emph{minimal subset} $S$ to review in \eqref{eq:enumeration_ccp} (and to include $\hat{\mathbf{S}}_\tau$ afterwards) as $\hat{S}(z) = \{s \in [N]\,|\,z_s=1\}$ where $z$ is a feasible solution of the system 
\begin{align}
 & \sum_{s=1}^N\,\pi^{(s)}\,z_s \geq 1-\tau & \label{eq:lower_support} \\ 
 & \sum_{s \,\in \,S} \, \pi^{(s)}\,z_s \leq 1-\tau-\check{\varepsilon} + \pi^{(\tilde{s})}\,+(1-z_{\tilde{s}}) &\quad \forall \tilde{s} \in [N] \label{eq:upper_support} \\
 & \sum_{s \,\in \,S} \, z_s \leq |S|-1& \quad \forall S \in \hat{\mathbf{S}}_\tau. \label{eq:discard_previous_S}
\end{align}
\end{remark}
\noindent First \eqref{eq:lower_support} and middle inequalities \eqref{eq:upper_support} above implement the requirements of \emph{minimal subsets} with $\check{\varepsilon} \in (0,\min_{s\in[N]}\, \pi^{(s)}/2]$ ensuring the strict inequality of the right-hand side of \eqref{eq:consistency_minimal_solution_def}.  One can notice that, indeed, if $z_s = 0$ (hence $s \not \in S=\hat{S}(z)$) the right-hand side of \eqref{eq:upper_support} is bigger than $1$ and the constraint is always valid. \\The last inequalities \eqref{eq:discard_previous_S} discard previously seen \emph{minimal subsets} $S$, stored in $\hat{\mathbf{S}}_\tau$. That being stated, the enumerative approaches described above quickly become impractical, even for moderate values of $N$. Noteworthy, when $\mathbf{S}_\tau$ is explicit (Example \ref{example:equiprobable}), the enumerative task becomes parallelisable across $B\geq 1$ units. \noindent Yet, unless $B=\Theta(|\mathbf{S}_\tau|)$, the minimal number of \eqref{eq:min_S_ccp} solved in series (hence the wall-clock elapsed time) grows exponentially with $N$.\\
\vspace{-7pt}

\noindent As previously announced, we close this section by posing \eqref{eq:new_min_problem} as the global minimization of $\mathcal{F}_\tau$, i.e. pointwise minimum of a finite collection of functions $f_S : \mathbb{R}^d \to \mathbb{R} \cup \{\infty\}$ for every $S \in \mathbf{S}_\tau$. The respective effective domains read 
$$ \text{dom}\,f_S = \mathcal{X}(S) := \bigcap_{s\,\in\,S}\,\Big(\mathcal{X}\,\cap\,\mathcal{R}(\xi^{(s)})\Big).$$
Thereby, one simply writes $F^*_\tau = \min_{x\in\mathbb{R}^d}\mathcal{F}_\tau(x)$ where, for every $x \in \mathbb{R}^d$,
\begin{equation}
\mathcal{F}_\tau(x)  := \min_{S\,\in\,\mathbf{S}_\tau}\,F(x)+\mathbb{I}_{\mathcal{X}(S)}(x) =   \min_{S\,\in\,\mathbf{S}_\tau}\,f_S(x). \label{eq:minimum_over_union}
\end{equation}\vspace{-9pt}
\begin{remark} In any situation, $\text{dom}\,\mathcal{F}$ is an union of $|\mathbf{S}_\tau|$ subsets, i.e. $$\text{dom}\,\mathcal{F}_\tau = \bigcup_{S \,\in\,\mathbf{S}_\tau}\,\mathcal{X}(S).$$ 
Let $\mathcal{X}$ be convex and $c(\cdot,\xi^{(s)})$ be quasi-convex for every $s \in [N]$. The domain of $\mathcal{F}_\tau$ stems as the union of finitely many convex sets. As analyzed in \cite{vanAckooij19} in the case where $\xi$ admits a log-concave continuous distribution, $\text{dom}\,\mathcal{F}_\tau$ might be convex as whole, depending on parameter $\tau$. Obviously, $\text{dom}\,\mathcal{F}_0$ is convex. In the continuous case, \cite{vanAckooij19} found that there exists a threshold $\hat{\tau}$ below which every $\tau\leq \hat{\tau}$ leads to a convex $\text{dom}\,\mathcal{F}_\tau$. It is sometimes true for discrete distributions. We illustrate this in Figure \ref{fig:pfl_experiment} where we corroborate the findings of \cite{vanAckooij19} with an empirical $\hat{\tau} \simeq 2\cdot10^{-2}$.
\vspace{-20pt}
\end{remark}
\begin{figure}[H]%
    \hspace{-20pt}
    \vspace{-25pt}
    \subfloat{{\includegraphics[scale=0.45]{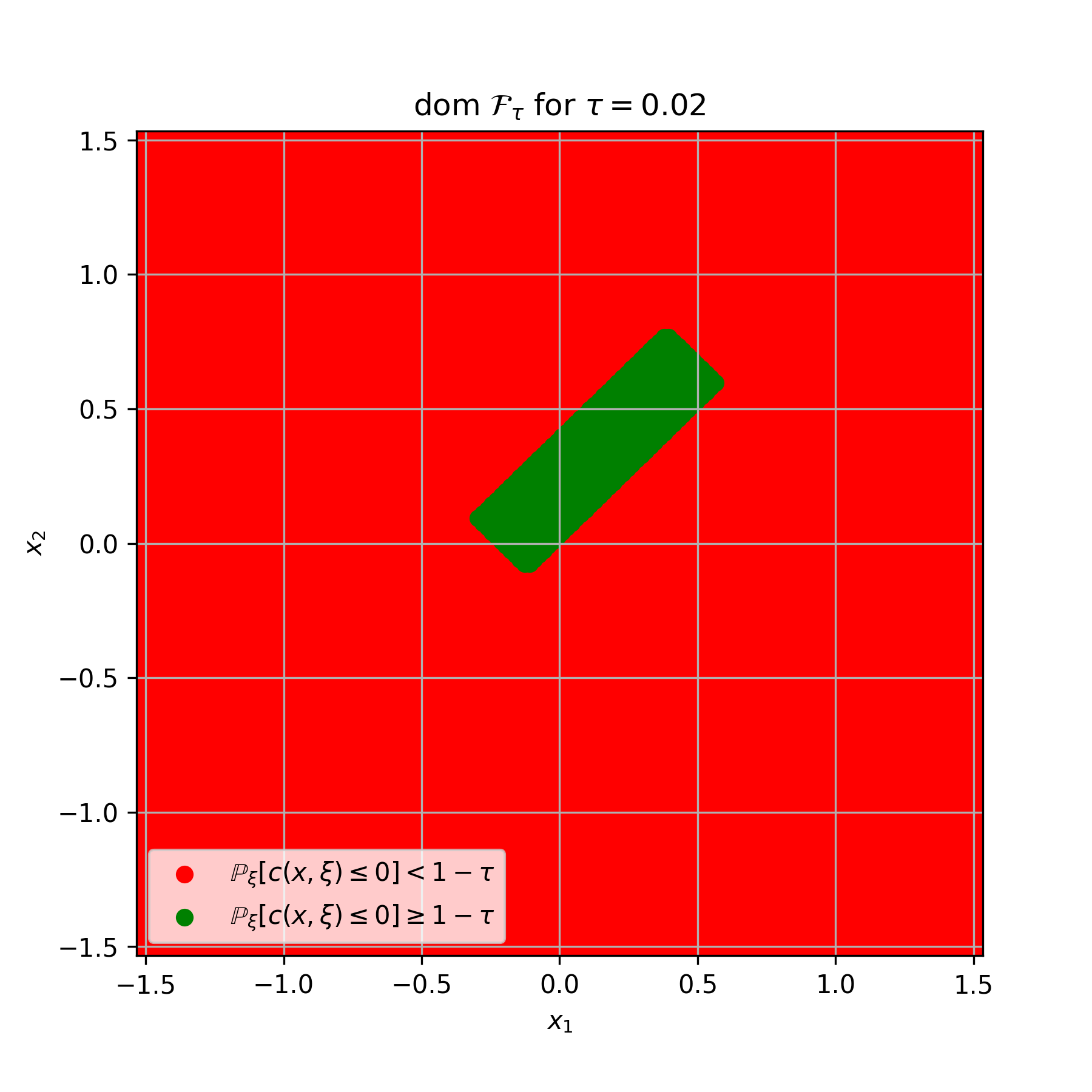} }}%
    \subfloat{{\includegraphics[scale=0.45]{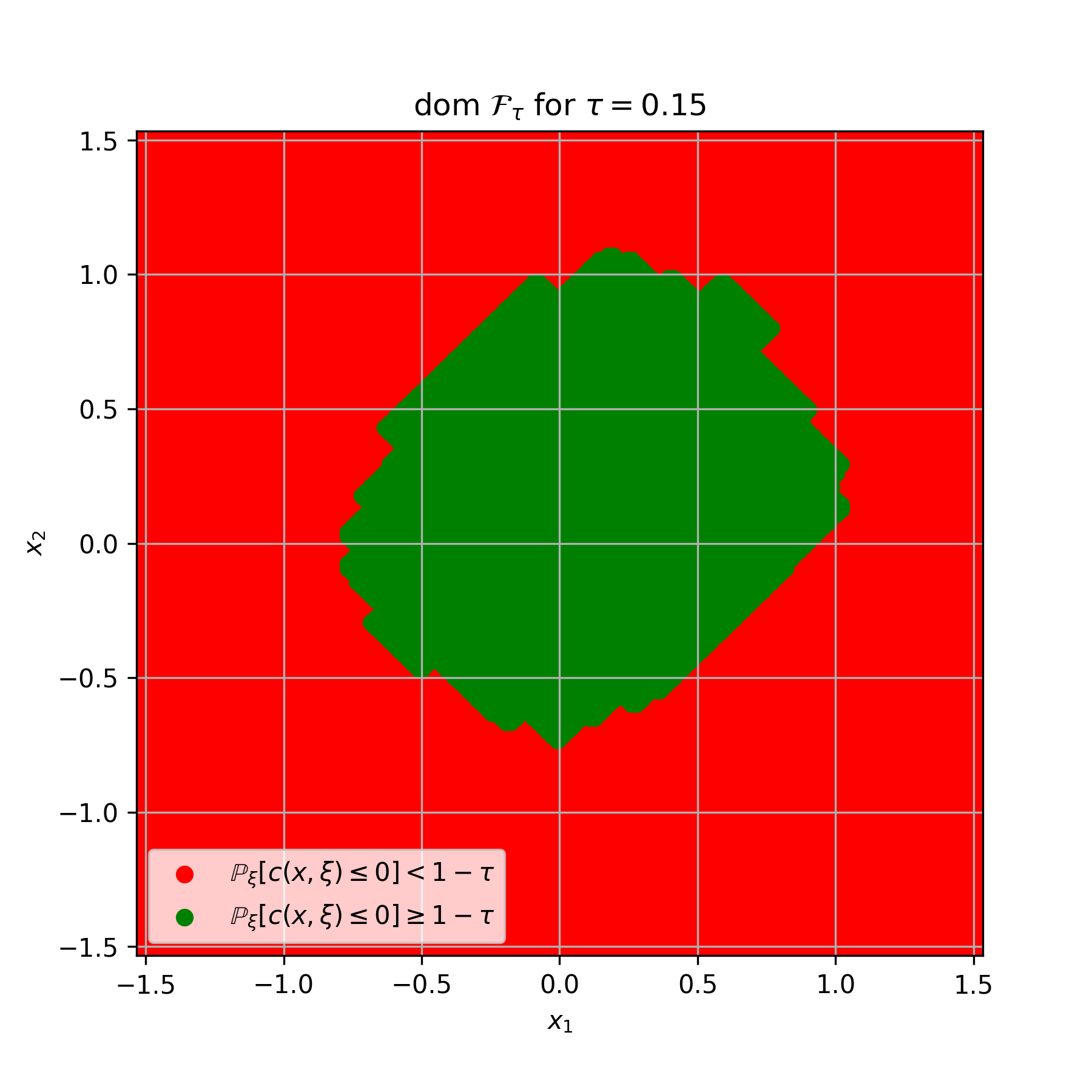} }}%
\vspace{15pt}
\caption{Probabilistic Facility Location (Example \ref{example:pfl}): convexity analysis of $\text{dom}\,\mathcal{F}_\tau$ | convex for every $\tau\leq 2\cdot10^{-1}$ (left) and nonconvex under $\tau=1.5\cdot10^{-1}$ (right).}
\label{fig:pfl_experiment}%
\end{figure}

\section{Presolve techniques} 
\label{sec:presolve_techniques}
Now that the background is all set, we can dive into our presolve techniques.\\ We emphasize the importance of Assumption (B) based on which we derive the useful Lemma \ref{lemma:consistency_quasi_ccp}. We recall that for any $z \in \{0,1\}^N$, $\hat{S}(z)=\{s \in [N]\,|\,z_s=1\}$. 

\begin{lemma}
\label{lemma:consistency_quasi_ccp}
Let $(x,z)$ be feasible for \eqref{eq:new_min_problem} and $\Xi = \textbf{conv}(\{\xi^{(\tilde{s})}\,|\,\tilde{s} \in \hat{S}(z)\})$. For every $s \in [N]$ such that $z_s=0$ and $\xi^{(s)}\in \Xi$, $(x,z+\mathbf{e}_s)$ stays feasible. 
\end{lemma}
\begin{proof}
By construction, $c(x,\xi^{(\tilde{s})})\leq 0$ for every $\tilde{s} \in [N]$ such that $z_{\tilde{s}}=1$. Let $\xi^{(s)} \in \Xi$ so that there exists weights $q^{(\tilde{s})}\in [0,1]$ summing up to $1$ with $\sum_{\tilde{s} \,\in\, \hat{S}(z)}\,q^{(\tilde{s})}\,\xi^{(\tilde{s})}=\xi^{(s)}.$ By quasi-convexity (Assumption (B)) of $c(x,\cdot)$,$$ c(x,\xi^{(s)}) \leq \max_{\tilde{s}\,\in\, \hat{S}(z)}\,c(x,\xi^{(\tilde{s})}) \leq 0.$$ 
It comes that if $z_s=0$ then it is possible to set $z_s=1$ without changing the value of $x$ hence certifying that $(x,z +\mathbf{e}_s)$ is a feasible solution of \eqref{eq:new_min_problem}. 
\end{proof}
\noindent Thus, one should consider feasible solutions $(x,z)$ of \emph{minimal volume} in the sense that there is no \emph{minimal subset} $S \subset \{s \in [N]\,|\,\xi^{(s)}\,\in\,\textbf{conv}(\{\xi^{(\tilde{s})}\,|\,\tilde{s}\in\hat{S}(z)\})\}$, \begin{equation}\textbf{conv}(\{\xi^{(s)}\,|\,s\in S\}) \subsetneq \textbf{conv}(\{\xi^{(\tilde{s})}\,|\,\tilde{s} \in \hat{S}(z)\}). \label{eq:consistency_ind} \end{equation} 
Otherwise, $(x,\hat{z})$ with $\hat{z}_s=1$ for every $s \in S$ and $\hat{z}_s=0$ is also feasible and furthermore, one has $F^*_\tau \leq \nu(S) \leq \nu(\hat{S}(z))$. When scenarios are equiprobable, for \emph{minimal volume} solutions, \emph{minimal subsets} $S=\hat{S}(z)$ must satisfy\begin{equation} 
s \in S \Leftrightarrow \xi^{(s)} \in \textbf{conv}(\{\xi^{(\tilde{s})}\,|\,\tilde{s} \in S\}).\label{eq:two_sides_cvx_incl}\end{equation} This is not true for non-uniform distributions as shown in Example \ref{example:simple_inc}.
\begin{example}\label{example:simple_inc}
Let $N=5$, $d=2$ and $\tau = 15/100$. For $s \in [4]$, $\pi^{(s)}=22/100$ and $\{\xi{(s)}\,|\,s\in[4]\} = \textbf{vert}(\mathbb{B}_\infty(\mathbf{0}_2,1))$ so that $\pi^{(5)}=1-4\cdot22/100=12/100$. The only possible \emph{minimal subset} is $S = [4] = \hat{S}(\mathbf{1}_5-\mathbf{e}_5)$. Then, if $\xi^{(5)} = \mathbf{0}_2$, $\xi^{(5)} \in \text{int}(\mathbb{B}_\infty(\mathbf{0}_2,1))$ yet $5\not \in S$.
\end{example}

\vspace{5pt}
\noindent We define now two types of subsets for the scenarios, i.e. \emph{safe} and \emph{pruned sets} (see Definition \ref{def:safe_set} \& \ref{def:pruned_set}), each of which serving as backbone of one or more techniques and allowing to compute upper/lower bounds. 

\begin{definition}[Safe set]\label{def:safe_set} $\oplus \subseteq [N]$ is a safe set if including the equalities $z_s=1$ for every $s \in \oplus$ in \eqref{eq:new_min_problem} does not increase its optimal value $F^*_\tau$.
\end{definition}
\begin{definition}[Pruned set]\label{def:pruned_set} $\ominus \subseteq [N]$ is a pruned set if including the equalities $z_s=0$ for every $s \in \ominus$ in \eqref{eq:new_min_problem} does not increase its optimal value $F^*_\tau$.
\end{definition}
\begin{remark}\label{remark:consistency_two_sides} As a direct consequence of both definitions, if one looks after a sound upper-bound like $\hat{F}^*_\tau=\nu(S) \geq F^*_\tau$, the subset $S\subseteq [N]$ should be such that 
\begin{equation}
S \in \mathbf{T}(\oplus,\ominus) := \Big\{ S \subseteq [N]\backslash \ominus\,|\,\oplus \subseteq S \wedge \sum_{s\,\in\,S}\,\pi^{(s)}\geq 1-\tau\Big\}.
    \label{eq:sound_subset_S}
\end{equation}
Similar to \eqref{eq:enumeration_ccp}, one derives $F^*_\tau = \min_{S\,\in\,\mathbf{T}(\oplus,\ominus)}\,\nu(S)$. Moreover, minimizing $F$ considering only constraints induced by $\oplus$ yields a lower-bound $ \check{F}^*_\tau=\nu(\oplus) \leq F^*_\tau$.
\end{remark}
\subsection{\emph{Safing} techniques} \label{subsec:safing}
In this subsection, we present sufficient conditions that help to build-up \emph{safe sets} incrementally. That is, starting from $\oplus=\emptyset$, \emph{selectable} indices $s \in [N]\backslash (\oplus\,\cup\,\ominus)$ can be tried\footnote{We emphasize that multiple indices can be processed at once, i.e. in parallel.} and if one of the conditions below is triggered, the \emph{safe set} is updated as $\oplus \gets \oplus\, \cup\, \{s\}$. We start by presenting a generic condition (Proposition \ref{prop:np_sel}) applying independently of Assumption (B). The second one (Proposition \ref{prop:ns_sel}), more involved, is new and specific to this work. Note however that, from a practical point of view, it can only be efficiently implemented when $p=2$ or $p=3$.\\

\noindent First, it happens that based on valid inequalities for \eqref{eq:new_min_problem}, the remaining feasible domain (or some relaxation of it) is small enough so that a given constraint $c(\cdot,\xi^{(s)})$ becomes non-positive everywhere therein. If so, $s \in \oplus$.

\begin{proposition}[Non-positivity selection]\label{prop:np_sel}
Let $\hat{F}^*_\tau \geq F^*_\tau$. For any $S \subseteq \oplus$,\[ 0 \geq \max_{x\,\in\,\mathcal{X},\,F(x) \leq F^*_\tau}\,c(x,\xi^{(s)}) \quad\text{s.t.}\quad c(x,\xi^{(\tilde{s})})\leq 0 \quad \forall \tilde{s}\in S \] implies that $s$ belongs to $\oplus$.
\end{proposition}

\noindent Second, conditional to the current states of sets $\oplus$ and $\ominus$, if $s$ belongs to every \emph{sound} subset $S \in \mathbf{T}(\oplus,\ominus)$ (see \eqref{eq:consistency_ind}), one can conclude that $s \in \oplus$. Through the lens of Lemma \ref{lemma:consistency_quasi_ccp}, this statement is equivalent to the following condition. 

\begin{proposition}[Non-separability induction]\label{prop:ns_sel} If there exists no $S \in \mathbf{T}(\oplus,\ominus)$ such that $\xi^{(s)}$ is \textbf{not} contained in $\Xi = \textbf{conv}(\xi^{(\tilde{s})}\,|\,\tilde{s} \in S)$ then $s \in \oplus$.
\end{proposition}
\begin{proof} We need to consider two outcomes. Either problem \eqref{eq:new_min_problem} is \emph{infeasible} and $F^*_\tau=\infty$ or there exists an optimal solution $(x^*,z^*)$ with $\hat{S}(z^*)\in \mathbf{T}(\oplus,\ominus)$. In the first case, it comes immediately that $s\in \oplus$ since, by definition, this equality will not affect the optimal value of the problem. Otherwise, if follows that either $z^*_s=1$ already or $z^*_s=0$ and $\xi^{(s)} \in \textbf{conv}(\xi^{(\tilde{s})}\,|\,\tilde{s} \in \hat{S}(z^*)\})$. In this latter situation, Lemma \ref{lemma:consistency_quasi_ccp} ensures that $(x^*,z^*+\mathbf{e}_s)$ stays \emph{feasible} hence optimal by hypothesis. Thereby, one can include $s$ in $\oplus$.
\end{proof}

\noindent Finally, we invoke Corollary \ref{coro:implied} that allows to expand a \emph{safe set} $\oplus$ with all the \emph{selectable} scenarios falling within the convex hull made of the scenarios of $\oplus$. 
\begin{corollary} The set $\{s \in [N]\,|\,\xi^{(s)} \in\,\textbf{conv}(\{\xi^{(\tilde{s})}\,|\,\tilde{s}\in \oplus \})\}$ is (also) safe. 
 \label{coro:implied} 
\end{corollary}
\begin{proof} The proof readily follows from Lemma \ref{lemma:consistency_quasi_ccp}.
\vspace{-15pt}
\end{proof}

\begin{figure}[H]%
    \hspace{-20pt}
    \vspace{-25pt}
    \subfloat{{\includegraphics[scale=0.45]{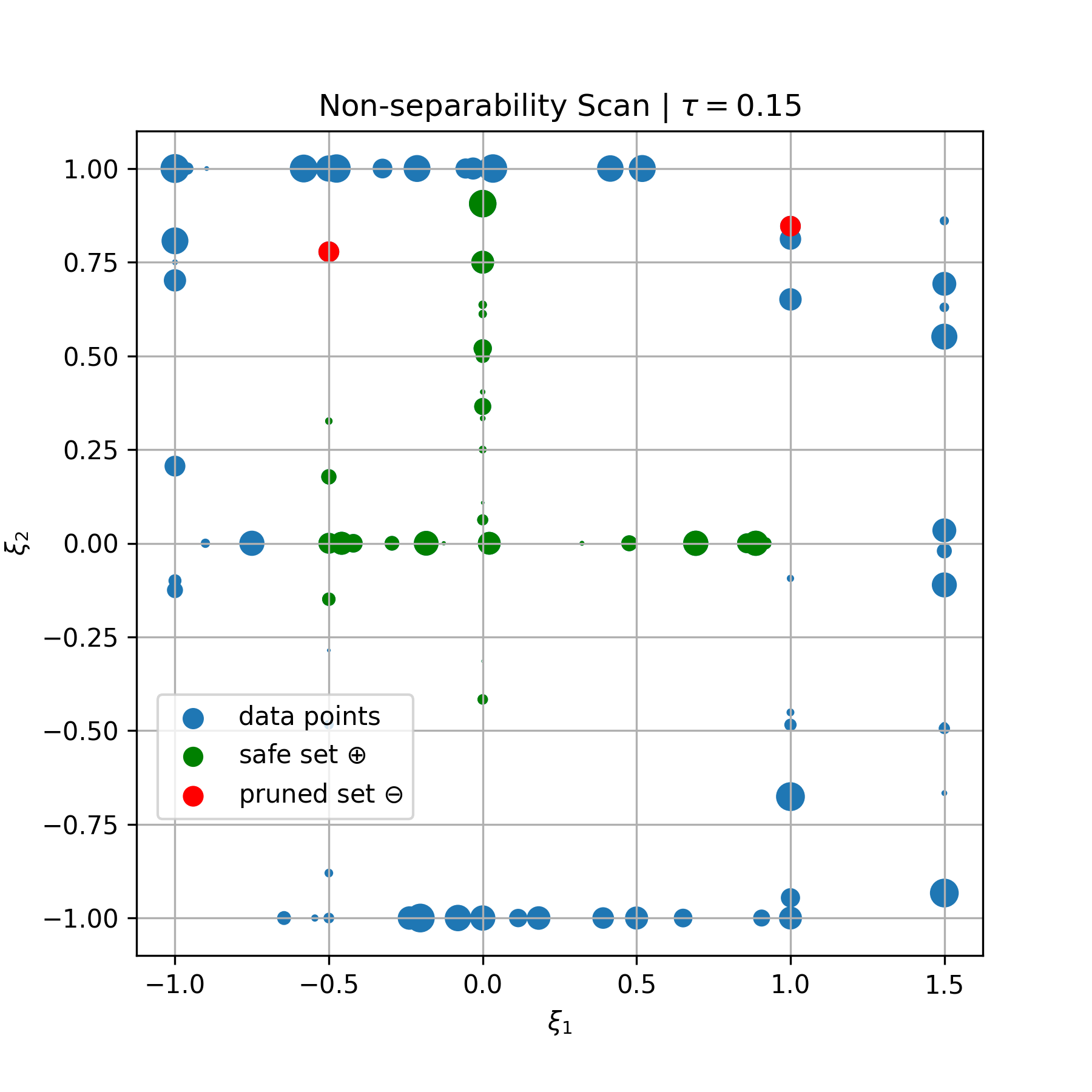} }}%
    \subfloat{{\includegraphics[scale=0.45]{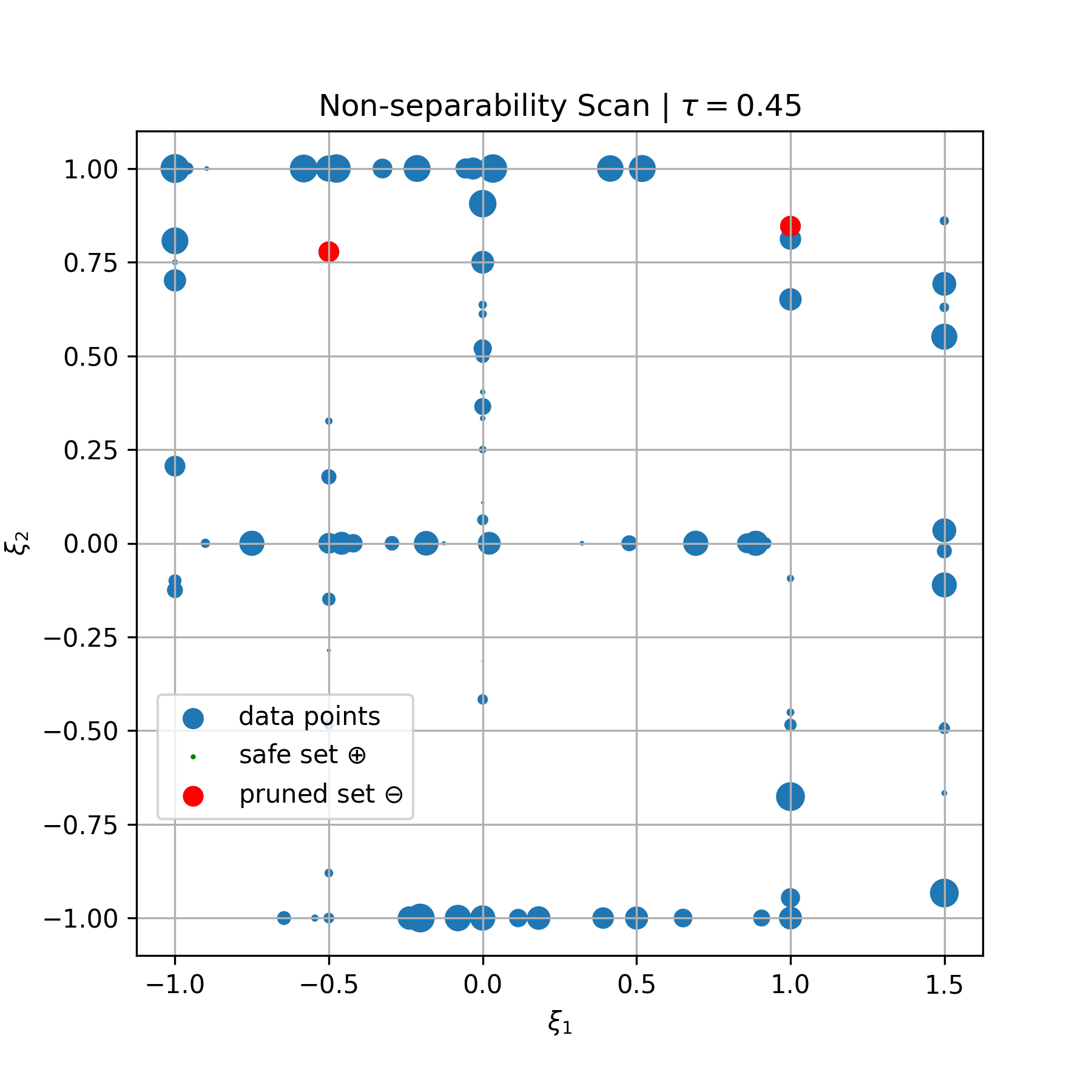} }}%
\vspace{15pt}
\caption{Probabilistic Facility Location (Example \ref{example:pfl}): \emph{safe set} incremental building based on Proposition \ref{prop:ns_sel} ; starting with $\oplus=\emptyset$ and $\ominus=\{11,74\}$ (toy example) | \\Init. $|\mathbf{T}(\oplus,\ominus)|$ smaller (left) than (right) hence more non-separable indices $s$.}
\label{fig:pfl_experiment}%
\vspace{-5pt}
\end{figure}
\subsection{\emph{Pruning} techniques}\label{subsec:pruning}
In the same spirit as Section \ref{subsec:safing}, yet diametrically opposite in goal, we aim now at discarding scenarios. Again, if one of the conditions below is satisfied, the \emph{pruned set} is updated as $\ominus \gets \ominus\, \cup\, \{s\}$.  \\

\noindent Akin to Proposition \ref{prop:np_sel}, it can happen, \emph{at contrario}, that the remaining feasible domain (or some relaxation of it) is small enough so that a given constraint $c(\cdot,\xi^{(s)})$ becomes strictly positive everywhere therein. If so, $s \in \ominus$.

\begin{proposition}[Strict-positivity exclusion]\label{prop:nn_excl}
Let $\hat{F}^*_\tau \geq F^*_\tau$. For any $S \subseteq \oplus$, 
$$ 0 < \min_{x\,\in\,\mathcal{X},\,F(x) \leq F^*_\tau}\,c(x,\xi^{(s)}) \quad\text{s.t.}\quad c(x,\xi^{(\tilde{s})})\leq 0 \quad \forall \tilde{s}\in S$$
implies that $s$ belongs to $\ominus$.
\end{proposition}

\noindent It is easy to observe that if a \emph{safe set} $\oplus$ entails enough probability weight, i.e. $\sum_{s\,\in\,\oplus}\,\pi^{(s)}\geq 1-\tau$  (e.g. a \emph{minimal subset}), it is optimal in the sense $\nu(\oplus)=F^*_\tau$.\\Otherwise, $\oplus$ must include (at least) one more index $s \in [N]\backslash (\oplus\,\cup\,\ominus)$. If optimizing \eqref{eq:new_min_problem} conditional to $z_{\tilde{s}}=1$ for every $\tilde{s} \in \oplus\, \cup\, \{s\}$  yields a value, i.e. $\nu(\oplus\, \cup\, \{s\})$,  falling strictly above $\hat{F}^*_\tau\geq F^*_\tau$ then the hypothesis that $z_s=1$ would not deteriorate the optimal value of \eqref{eq:new_min_problem} is rejected.
\begin{proposition}[Sub-optimality exclusion]\label{prop:subopt_excl}
Let $\hat{F}^*_\tau \geq F^*_\tau$ and let $\oplus \subseteq [N]$ be a non-optimal safe set . If $\nu(\oplus \cup \{s\})>\hat{F}^*_\tau$ then $s \in \ominus$. Moreover, it holds 
\begin{equation} \check{F}^*_\tau = \min_{s\, \in\, [N]\backslash (\oplus\, \cup\, \ominus)}\, \nu(\oplus\, \cup\, \{s\}) \leq F^*_\tau.\label{eq:often_useless_LB}\end{equation}
\end{proposition}

\subsection{Valid inequalities}
In some situations, e.g. constraints are decomposable as $c(x,\xi) = \bar{c}(x)-\xi$, it is possible to define so-called \emph{(generalized) precedence constraints} (see \cite{Ruszczynski02} and references therein). That is, one can define a partial order $\preceq$ in the $\xi$-space (e.g. elementwise comparison $\leq$) so that from $\xi^{(s_1)}\preceq \xi^{(s_2)}$, $z_{s_1}\leq z_{s_2}$ becomes a valid inequality for \eqref{eq:new_min_problem}. Here, we propose two kinds of valid inequalities, deduced from Lemma \ref{lemma:consistency_quasi_ccp}. The proofs are omitted due to space constraints.   \\

\noindent Let $S\subseteq[N]$ and $\Xi= \textbf{conv}(\xi^{(\tilde{s})}\, |\, \tilde{s}\in S\})$. We denote by $\mathbf{V}(S)$ indices of scenarios that are vertices for $\Xi$, i.e. $\mathbf{V}(S) := \{s\in S\, |\, \xi^{(s)}\in \textbf{vert}(\Xi)\}$ and we depict by $\mathbf{N}(S)$ indices of all the scenarios included in $\Xi$, i.e. $\mathbf{N}(S) := \{s\in S\, |\, \xi^{(s)}\in \Xi\}$.

\begin{proposition}[Convex-hull inductions]\label{prop:cvx_hull_induction} 
The following inequalities are valid 
\begin{equation}
z_s \geq \sum_{\tilde{s}\in \mathbf{V}(S)} z_{\tilde{s}}-|\mathbf{V}(S)|+1 \quad\forall s \in \mathbf{N}(S) \backslash \mathbf{V}(S).\label{eq:vi_1}
\end{equation}
If scenarios are equally as likely and $\sum_{s\in \mathbf{N}(S)}\,\pi^{(s)} \geq (1-\tau)+N^{-1}$ then
\begin{equation}
\sum_{\tilde{s}\in \mathbf{V}(S)} z_{\tilde{s}} \leq |\mathbf{V}(S)|-1 .\label{eq:vi_2}
\end{equation}
\end{proposition}

\newpage 
\section{Numerical Experiments}
\label{sec:practice}
We focus here on the problem of Example \ref{example:pfl} (and its straightforward generalization for $p=3$) to showcase the benefits of the following presolve routine. \\
\noindent We initialize $(\oplus, \ominus) =(\emptyset, \emptyset)$ and $(\check{F}^*_\tau,\hat{F}^*_\tau) = (\infty,\infty)$.
\vspace{-1pt}
\begin{enumerate}
\item For every $s\in[N]$, we solve problem $\nu(\{s\})$, record its minimizer $x^*(\{s\})$, update $\check{F}^*_\tau \gets \min\{\check{F}^*_\tau,\nu(\{s\})\}$ and,  if $x^*(\{s\}) \in \text{dom}\,\mathcal{F}_\tau$, $\hat{F}^*_\tau \gets \min\{\hat{F}^*_\tau,\nu(\{s\})\}$.\\
At the end, $\check{F}^*_\tau\leq F^*_\tau\leq \hat{F}^*_\tau$. If $\hat{F}^*_\tau<\infty$, we set $\ominus \gets \{s\in [N]\,|\,\nu(\{s\})>\hat{F}^*_\tau\}$.
\item For every $s \in [N]\backslash (\oplus \cup \ominus)$, condition of Proposition \ref{prop:ns_sel} is checked to possibly add $s$ in $\oplus$. It is advised to set a time limit after which one stops the checks. Indeed, our implementation of the \emph{non-seperability induction} check is of complexity $\mathcal{O}(N^3)$ for $p=3$ compared to $\mathcal{O}(N\log N)$ for $p=2$. Thus, we chose $60$[s] (never reached in practice) for $p=2$ and $120$[s] for $p=3$.
\item We apply Corollary \ref{coro:implied} (possibly extending $\oplus$) and we implement \emph{sub-optimality exclusion} checks for every $s \in [N]\backslash (\oplus \cup \ominus)$ (possibly extending $\ominus$).
\item We finish by tightening big-$M$ bounds as in \eqref{eq:refined_bigMs_safe}.
\vspace{-7pt}
\end{enumerate}
\subsubsection*{Benchmarks} All the details about our implementations as well as the data generation are freely available on
\href{https://github.com/guiguiom/CCP_}{GitHub}. The parameters in \eqref{eq:pfl_ccp} were taken as $p \in \{2, 3\}$, $o=2$,  $\tilde{o}=1$ and $R = \frac{23}{25}\cdot p \cdot \max_{s\, \in\, [N]}\, \|\xi^{(s)}\|_\infty$. For each level $\tau \in \{0.05, 0.15\}$, $5$ independent random datasets $\mathbf{D}$ were drawn and fed the model. An overall time limit was set to $500$[s] ($p=2$) and $720$[s] ($p=3$) to solve the problem. We report below the results: \textbf{UB} (respectively \textbf{LB}) stands as the best upper-bound (respectively lower-bound) found during a solve.
\vspace{-10pt}
\begin{table}
\begin{tabular}{l||c|c|c|c|c|c}
\emph{method} & p & $\tau$ & avg. time [s] | $\textbf{UB}=\textbf{LB}$ & \# insts. solved & avg. $(\textbf{UB}-F^*_\tau)/F^*_\tau$\\
 \hline
 \hline
\texttt{presolve}  & \multirow{2}{*}{$2$}& \multirow{2}{*}{$5\%$} & \textbf{211.75} & \textcolor{green}{5/5} & $0\%$ \\
\texttt{direct} & &  &321.07 & \textcolor{purple}{3/5} & $1.15\%$\\ 
\hline 
\texttt{presolve}  & \multirow{2}{*}{$2$}  & \multirow{2}{*}{$15\%$}  & \textbf{157.41} & \textcolor{green}{5/5} & $0\%$\\
\texttt{direct} & &  &(time limit: 500)& \textcolor{red}{0/5} & $110.3\%$ (excl. $1$ outlier)\\ 
\hline
\texttt{presolve}  & \multirow{2}{*}{$3$}& \multirow{2}{*}{$5\%$} & \textbf{438.91} & \textcolor{blue}{4/5} & $\leq 0.8\%$ \\
\texttt{direct} & &  &(time limit: 720) & \textcolor{red}{0/5} & $14.31\%$\\ 
\hline 
\texttt{presolve}  & \multirow{2}{*}{$3$}  & \multirow{2}{*}{$15\%$}  & \textbf{363.24} & \textcolor{blue}{3/5} & $\leq 7.49\%$ \\
\texttt{direct} & &  &(time limit: 720)& \textcolor{red}{0/5} &$122.8\%$ \\ 
\hline
\end{tabular}
\caption{Comparison between use of presolve and direct resolution}
\end{table}
\vspace{-30pt}

\subsubsection{Comments} The fourth column represents the mean time taken by an algorithm, either \texttt{presolve} or \texttt{direct}, when it did success in providing the optimal solution (i.e. $\textbf{UB}=\textbf{LB}=F^*_\tau$). The last column represents the average \emph{a posteriori} relative optimality gap achieved. Note that if the optimal value $F^*_\tau$ of a given instance was not found within time limit for both \emph{methods}, we launched again \texttt{presolve} for a long period of time until we obtain it. As clearly shown in the results reported above, our preliminary numerical experiments demonstrate that the use of our presolve techniques is computationally beneficial.   
 We leave as future research the possibility to incorporate our valid inequalities \eqref{eq:vi_1} \& \eqref{eq:vi_2} in \eqref{eq:new_min_problem} and to extend the applicability of our methodology to higher dimensions $p>3$.
\bibliographystyle{splncs04}
\bibliography{samplepaper.bib}
\end{document}